\documentclass[10pt,reqno]{amsart}
\usepackage{floatflt}
\usepackage{xypic}
\usepackage{rotating}
\usepackage{bbm}
\usepackage{mathrsfs}
\usepackage{diagbox}
\usepackage{dutchcal}
\usepackage{extarrows}
\usepackage{cite}
\usepackage{amsfonts} 
\usepackage[dvipsnames,usenames]{color}
\usepackage{graphicx,latexsym,bm,amsmath,amssymb,verbatim,multicol,lscape}
\makeatletter
\@namedef{subjclassname@2020}{
\textup{2020} Mathematics Subject Classification}
\makeatother
\newtheorem{theorem}{Theorem} [section]

\newtheorem{conjecture}[theorem]{Conjecture}
\newtheorem{lemma}[theorem]{Lemma}

\numberwithin{equation}{section}

\def\bew{\begin{widetext}}
\def\eew{\end{widetext}}
\def\be{\begin{equation}}
\def\ee{\end{equation}}
\def\bea{\begin{eqnarray}}
\def\eea{\end{eqnarray}}

\allowdisplaybreaks
\begin{document}
\title[Divisibility among power GCD matrices and power LCM matrices]
{Gcd-closed sets and divisibility among power GCD matrices and power LCM matrices}

\begin{abstract}
Let $a, b$ and $n$ be positive integers and let $S=\{x_1, \cdots, x_n\}$
be a set of $n$ distinct positive integers. We denote by $(S^a)$
(resp. $[S^a]$) the $n\times n$ matrix having the $a$th power of the
greatest common divisor (resp. the least common multiple) of $x_i$ and
$x_j$ as its $(i,j)$-entry. In 1875, Smith published his famous theorem
stating that if $S$ is an
FC set, then $\det(S^a)=\prod_{k=1}^nJ_a(x_k)$, where $J_a$ is the Jordan's
totient function. In 1993, Bourque and Ligh gave an explicit formula
for $\det[S^a]$ when $S$ is gcd closed (i.e. $\gcd(x_i, x_j)\in S$
for all integers $i$ and $j$ with $1\le i,j\le n$), and showed in 1995 that
the $a$th power GCD matrix $(S^a)$ divides the $a$th power LCM matrix
$[S^a]$ (written as $(S^a)|[S^a]$) in the ring $M_n(\mathbb Z)$ of
$n\times n$ matrices over the integers when $S$ is FC. In 2002,
Hong proved that such factorization is no longer true when $S$ is gcd closed.
In 2008 and 2026,
Hong showed that $(S^a)|(S^b)$ and $(S^a)|[S^b]$ if $S$ is a divisor chain
and an FC set, respectively. For any $x\in S$, let $G_{S}(x)
:=\{d\in S: d<x, d|x \ {\rm and} \ (d|y|x, y\in S)\Rightarrow y\in \{d,x\}\}$.
Hong proved in 2026 that any FC set is a gcd-closed set satisfying the
condition $\mathcal{G}$ (i.e., for any element $x\in S$, either $G_S(x)$
contains at most one element, or $G_S(x)$ contains at least two elements
and satisfies that $[y_1,y_2]=x$ and $(y_1,y_2)\in G_S(y_1)\cap G_S(y_2)$
for any $\{y_1,y_2\}\subseteq G_S(x)$). In this paper, we first establish
some interesting and important arithmetic properties of gcd-closed sets
satisfying the condition $\mathcal{G}$, and then make use of these results
to show that $(S^a)|(S^b)$ and $(S^a)|[S^b]$ when $a|b$ and $S$ is a gcd-closed
set satisfying the condition $\mathcal{G}$. This confirms a conjecture of
Hong proposed in [S.F. Hong, Divisibility among power GCD matrices and
power LCM matrices, {\it Bull. Aust. Math. Soc.} {\bf 113} (2026), 231-243].
\end{abstract}

\author[G.Y. Zhu]{Guangyan Zhu}
\address{School of Mathematics and Statistics,
Hubei Minzu University, Enshi 445000, P.R. China}
\email{2009043@hbmzu.edu.cn}
\thanks{The research was supported in part by the Startup
Research Fund of Hubei Minzu University for Doctoral Scholars
(Grant No. BS25008)}
\maketitle
		
\section{Introduction}
For arbitrary integers $x$ and $y$, we denote by $(x, y)$ (resp. $[x, y]$)
their greatest common divisor (resp. least common multiple).
Let $\mathbb Z$ denote the ring of integers and $\mathbb Z^+$
the set of positive integers.
Let $|T|$ stand for the cardinality of a finite set $T$ of integers.
Let $f$ be an arithmetic function and let $S=\{x_1, \cdots, x_n\}$
be a set of $n$ distinct positive integers. Let $(f(x_i, x_j))$
(abbreviated by $(f(S))$) denote the $n\times n$ matrix having $f$
evaluated at the greatest common divisor $(x_i, x_j)$ of $x_i$ and
$x_j$ as its $(i,j)$-entry. Let $(f[x_i, x_j])$ (abbreviated by
$(f[S])$) denote the $n\times n$ matrix having $f$ evaluated at the
least common multiple $[x_i, x_j]$ of $x_i$ and $x_j$ as its
$(i,j)$-entry. Let $\xi_a$ be the arithmetic function defined by $\xi_a(x)=x^a$
for any positive integer $x$, where $a$ is a positive integer.
The $n\times n$ matrix $(\xi_a(x_i, x_j))$ (abbreviated by $(S^a)$)
and $(\xi_b[x_i, x_j])$ (abbreviated by $[S^b]$) are called {\it power GCD matrix}
on $S$ and {\it power LCM matrix} on $S$, respectively. The set $S$ is called
{\it factor closed} (FC) if $d\in S$ whenever $x\in S$ and $d|x$. In 1875,
Smith {\cite{[S-PLMS1875]}} published his famous theorem by showing that if
$S$ is an FC set, then $\det(S^a)=\prod_{k=1}^nJ_a(x_k),$
where $J_a:=\xi_a*\mu$ is the Jordan's totient function,
$\mu$ is the M\"{o}bius function and $\xi_a*\mu$ is the Dirichlet
convolution of $\xi_a$ and $\mu$. Since then, lots of
generalizations of Smith's determinant and related results
have been published (see, for example, \cite{[A-PJM1972]}-\cite{[NZ-1980]}
and \cite{[Y-RJ2010]}-\cite{[ZY-2025]}).

We say that the set $S$ is {\it gcd closed} if $(x_i, x_j)\in S$
for all $1\le i,j\le n$. Evidently, any FC set is gcd closed
but not conversely. In 1993, Bourque and Ligh \cite{[BL-LMA1993]}
extended Smith's theorem and the Beslin-Ligh result \cite{[BL-BAMS1989]}
by showing that if $S$ is gcd closed, then
$\det(S^a)=\prod_{k=1}^n \alpha_{\xi_a}(x_k)$, where
\begin{equation}\label{eq 1.1}
\alpha_{\xi_a}(x_k):=\sum_{d|x_k\atop d\nmid x_t,x_t<x_k}(\xi_a*\mu)(d).
\end{equation}
Another interesting and remarkable result was obtained in \cite{[HHL-AMH2016]}
and states that if $n\ge 2$ and $S=\{2,3,\cdots,n\}$, then
$$
\det(S^a)=\Big(\prod_{k=1}^nJ_a(k)\Big)
\sum_{t=1\atop t\ {\tiny\hbox{is squarefree}}}^n\frac{1}{J_a(t)},
$$
where a positive integer is called {\it squarefree} if it is
divisible by no other perfect square than 1. The first ten
squarefree integers are given as follows:
$1,2,3,5,6,7,10,11,13,14$.

Divisibility is the most important topic in the field of
Smith matrices. Bourque and Ligh \cite{[BL-LAA1992]} showed that
if $S$ is FC, then $(S)$ divides $[S]$ (written as $(S)|[S]$)
in the ring $M_n(\mathbb Z)$ of $n\times n$ matrices over the integers.
That is, there exists an
$A \in M_n(\mathbb Z)$ such that $[S]=(S)A$ or $[S]=A(S)$. Hong
\cite{[H-LAA2002]} showed that such a factorization is no longer
true in general if $S$ is gcd closed. Let $x,y \in S$ with
$y<x$. If $y\mid x$ and the conditions $y\mid d\mid x$ and $d\in S$
imply that $d\in \{x,y\}$, then we say that $y$ is
a {\it greatest-type divisor} of $x$ in $S$. For $x\in S$,
$G_S(x)$ stands for the set of all greatest-type divisors
of $x$ in $S$. The concept of greatest-type divisor
was introduced by Hong and played a key role in his
solution \cite{[H-AMC1996],[H-JA1999]} of the Bourque-Ligh conjecture
\cite{[BL-LAA1992]}. By \cite{[H-LAA2002]}, we know that there are
gcd-closed sets $S$ with $\mathop {\max}_{x \in S}
\left\{ {\left| {G_S\left(x\right)}\right|}\right\}=2$
such that $(S)^{-1}[S]\notin M_n(\mathbb Z)$.

In 2008, Hong \cite{[H-LAA2008]} initiated the study of the divisibility
properties among power GCD matrices and among power LCM matrices by showing
that if $a|b$ and $S$ is a divisor chain, then $(S^a)|(S^b)$ and $(S^a)|[S^b]$
in the ring $M_{n}(\mathbb Z)$. In 2022, Zhu \cite{[Z-IJNT2022]} showed that
the same divisibility results are true when $a|b$ and $S$ is a gcd-closed set
with $\max_{x\in S}\{|G_S (x)|\}=1$. In 2026, Hong \cite{[H-BAMS25]} established
the same results when $a|b$ and $S$ is FC. As in \cite{[H-BAMS25]}, for any
set $S$ of positive integers and for any $x\in S$ with $|G_S(x)|\ge 2$,
we say that the two distinct greatest-type
divisors $y_1$ and $y_2$ of $x$ in $S$ {\it satisfy the condition
$\mathcal{G}$} if $[y_1,y_2]=x$ and $(y_1,y_2)\in G_S(y_1)\cap G_S(y_2)$,
and we say that $x$ {\it satisfies the condition $\mathcal{G}$} if any two
distinct greatest-type divisors of $x$ in $S$ satisfy the condition $\mathcal{G}$.
Moreover, we say that a set $S$ of positive integers {\it satisfies the
condition $\mathcal{G}$} if any element $x$ in $S$ satisfies that either
$|G_S(x)|\le 1$, or $|G_S(x)|\ge 2$ and $x$ satisfies the condition $\mathcal{G}$.
In \cite{[H-BAMS25]}, Hong showed that any FC set is a gcd-closed
set satisfying the condition $\mathcal{G}$. Furthermore, Hong proposed
the following interesting conjecture.

\begin{conjecture}{\rm {\cite[Conjecture 3.4]{[H-BAMS25]}}} \label{p1.1}
Let $a$ and $b$ be positive integers with $a|b$ and let $S$ be
a gcd-closed set satisfying the condition $\mathcal G$. Then
$(S^a)|(S^b)$ and $(S^a)|[S^b]$ in the ring $M_{|S|}(\mathbb Z)$.
\end{conjecture}

For the case $\max_{x\in S}\{|G_S (x)|\}=1$, by Zhu's theorem
\cite{[Z-IJNT2022]} we know that Conjecture \ref{p1.1} is true.
For the case $\max_{x\in S}\{|G_S (x)|\}=2$, by \cite{[WZ]}
one knows that Conjecture \ref{p1.1} holds. It was remarked that
Wan and Zhu \cite{[WZ]} showed the existence of gcd-closed sets $S$
with $\max_{x\in S}\{|G_{S}(x)|\}=2$ and the condition $\mathcal{G}$
not being satisfied and infinitely many integers $b\ge 2$
such that $(S)\mid (S^b)$ (resp. $(S)\mid [S^b]$ and $[S]\mid [S^b]$).
In \cite{[ZLW]}, Zhu, Luo and Wan confirmed Conjecture \ref{p1.1}
for the case $\max_{x\in S}\{|G_S (x)|\}=3$. However, Conjecture
\ref{p1.1} still remains open for the general gcd-closed set $S$
satisfying the condition $\mathcal{G}$.

In this paper, our main goal is to study the divisibility among
power GCD matrices and power LCM matrices. We introduce a new
method to give a unified treatment of Conjecture \ref{p1.1}.
Actually, we will exploit arithmetic properties of
gcd-closed sets satisfying the condition $\mathcal{G}$.
Then we make use of these results to show that
Conjecture \ref{p1.1} is true. The main result of this paper
can be stated as follows.

\begin{theorem}\label{theorem 1.3}
Let $S$ be a gcd-closed set satisfying the condition
$\mathcal{G}$  and let $a$ and $b$ be positive integers
such that $a|b$. Then the $a$th power GCD matrix $(S^a)$
divides both the $b$th power GCD matrix $(S^b)$
and the $b$th power LCM matrix $[S^b]$ in the ring
$M_{|S|}(\mathbb Z)$.
\end{theorem}

This paper is organized as follows. In Section 2, we investigate
the arithmetic properties of gcd-closed sets that satisfy the
condition $\mathcal G$. Section 3 is devoted to reducing the
coefficients $c_{ij}$ and $\alpha_{\xi_a}(x_m)$. The proof of
Theorem \ref{theorem 1.3} is presented in Section 4.

One can easily check that for any permutation $\sigma$
on the set $\{1,\cdots,n\}$, $(S^a)|(S^b)\Leftrightarrow
(S_{\sigma}^a)|(S_{\sigma} ^b)$ and $(S^a)|[S^b]\Leftrightarrow
(S_{\sigma}^a)|[S_{\sigma} ^b]$, where
$S_\sigma:=\{x_{\sigma (1)},\cdots, x_{\sigma (n)}\}$.
So without loss of any generality, we can always assume that the
set $S=\{x_1,\cdots, x_n\}$ satisfies that $x_1<\cdots <x_n$.

\section{Arithmetic properties of gcd-closed sets satisfying condition $\mathcal G$}

Throughout this section, let $x_m \in S$ with
$G_S(x_m)=\{x_{m_1},\cdots,x_{m_s}\}$ and
$1\le s\le |S|-1$. Let $1\le k\le s$ and for any
integers $i_1,\cdots,i_k$ with
$1\le i_1<\cdots<i_k\le s$, we define
$$x_{m_{i_1\cdots i_k}}:=(x_{m_{i_1}},\cdots, x_{m_{i_k}}).$$
Then $x_{m_{1\cdots s}}=(x_{m_1},\cdots, x_{m_s})$.
Define $\langle s\rangle:=\{1,\cdots,s\}$.
In what follows, we analyze the arithmetic properties
of the gcd-closed set $S$ which satisfies the condition
$\mathcal{G}$. For any $j\in\langle s\rangle$, define
$P_j:=\frac{x_m}{x_{m_j}}$. We begin with a known result
that is due to Zhu and Yu \cite{[ZY-2025]}.

\begin{lemma}\label{lemma 2.3} {\rm\cite{[ZY-2025]}}
Let $S$ be a gcd-closed set satisfying the condition
$\mathcal{G}$ and $x_m \in S$. Let $k$ be an integer
with $1 \le k \le s$. Then for any integers
$i_1,\cdots,i_k$ with $1\le i_1<\cdots<i_k\le s$, we have
$x_{m_{i_1\cdots i_k}} P_{i_1}\cdots P_{i_k}=x_m$.	
\end{lemma}

\begin{lemma}\label{lemma1}
Let $S$ be a gcd-closed set satisfying the condition
$\mathcal{G}$ with $x_m \in S$. Let $|G_S(x_m)|\ge 2$
and $\{x_{m_{i_1}},\cdots,x_{m_{i_k}}\}
\subseteq G_S(x_m)$ with $1\le k\le |G_S(x_m)|-1$. Then
$x_{m_{i_1\cdots i_kj}}\in G_S(x_{m_{i_1\cdots i_k}})$
for any $x_{m_j}\in G_S(x_m)
\setminus \{x_{m_{i_1}},\cdots, x_{m_{i_k}}\}$.
\end{lemma}
\begin{proof}
We prove Lemma \ref{lemma1} by using induction on $k$.
Let $k=1$. Since $x_m\in S$ satisfies the condition $\mathcal G$
and $\{x_{m_{i_1}},x_{m_j}\}\subseteq G_S(x_m)$
for any $x_{m_j}\in G_S(x_m)\setminus \{x_{m_{i_1}}\}$, we have
\begin{align}\label{q1}
x_{m_{i_1}j}=(x_{m_{i_1}},x_{m_j})\in G_S(x_{m_{i_1}}).
\end{align}
as desired. The statement is true in this case.

Let $k=2$. Since $x_m$ satisfies the condition $\mathcal G$,
by \eqref{q1} we get that $x_{m_{i_1i_2}}\in G_S(x_{m_{i_1}})$
and $x_{m_{i_1}j}\in G_S(x_{m_{i_1}})$
for any $x_{m_j}\in G_S(x_m)\setminus \{x_{m_{i_1}},x_{m_{i_2}}\}$.
It then follows that $x_{m_{i_1i_2j}}=(x_{m_{i_1i_2}}, x_{m_{i_1}j})
\in G_S(x_{m_{i_1i_2}})$ since $x_{m_{i_1}}\in S$ satisfies
the condition $\mathcal G$. The statement holds in this case.

Assume that the statement is true for the $k$ case. Let
$\{x_{m_{i_1}},\cdots,x_{m_{i_k}},x_{m_{i_{k+1}}}\}
\subseteq G_S(x_m)$. Pick an arbitrary elment $x_{m_j}\in G_S(x_m)
\setminus \{x_{m_{i_1}},\cdots, x_{m_{i_k}},x_{m_{i_{k+1}}}\}$.
Then $\{x_{m_{i_{k+1}}},x_{m_j}\}\subseteq G_S(x_m)
\setminus \{x_{m_{i_1}},\cdots, x_{m_{i_k}}\}.$
Then by the inductive hypothesis, we have
$\{x_{m_{i_1\cdots i_ki_{k+1}}},x_{m_{i_1\cdots i_kj}}\}
\subseteq G_S(x_{m_{i_1\cdots i_k}})$. Since
$x_{m_{i_1\cdots i_k}}\in S$ satisfies the condition
$\mathcal G$, it follows that
$x_{m_{i_1\cdots i_ki_{k+1}j}}
=(x_{m_{i_1\cdots i_ki_{k+1}}}, x_{m_{i_1\cdots i_kj}})
\in G_S(x_{m_{i_1\cdots i_ki_{k+1}}}).$
Thus the statement holds for the $k+1$ case.

This finishes the proof of Lemma \ref{lemma1}.
\end{proof}

\begin{lemma}\label{lemma 2.2}
Let $S$ be a gcd-closed set satisfying the condition
$\mathcal{G}$ and $x_m \in S$.
Then $(P_i,P_j)=1$ for $1 \leq i\ne j \leq s$.
\end{lemma}
\begin{proof}
For $1 \leq i\ne j \leq s$, let $x_{m_i}=g_ix_{m_{ij}}$	
and $x_{m_j}=g_jx_{m_{ij}}$. Since $(x_{m_i},x_{m_j})=x_{m_{ij}}$,
one has $(g_i,g_j)=1$. From the assumption that $x_m\in S$
satisfies the condition $\mathcal G$ and
$\{x_{m_i}, x_{m_j}\}\subseteq G_S(x_m)$,
one derives that $[x_{m_i}, x_{m_j}]=x_m$. This together
with the fact $(x_{m_i}, x_{m_j})[x_{m_i}, x_{m_j}]=x_{m_i}x_{m_j}$
yields $x_{m_{ij}}x_m=x_{m_i}x_{m_j}$. Thus we obtain that
\begin{align*}
(P_i,P_j)=&\Big(\frac{x_m}{x_{m_i}}, \frac{x_m}{x_{m_j}}\Big)
=\frac{\frac{x_m}{x_{m_i}}\cdot\frac{x_m}{x_{m_j}}}
{\Big[\frac{x_m}{x_{m_i}}, \frac{x_m}{x_{m_j}}\Big]}
=\frac{\frac{x_m}{x_{m_i}}\cdot\frac{x_m}{x_{m_j}}}
{\Big[\frac{x_m}{g_ix_{m_{ij}}},\frac{x_m}{g_jx_{m_{ij}}}\Big]}\\
=& \frac{g_ig_jx_{m_{ij}}\cdot\frac{x_m}{x_{m_i}}\cdot\frac{x_m}{x_{m_j}}}
{g_ig_jx_{m_{ij}}\Big[\frac{x_m}{g_ix_{m_{ij}}},\frac{x_m}{g_jx_{m_{ij}}}\Big]}
=\frac{g_ig_jx_{m_{ij}}\cdot\frac{x_m}{x_{m_i}}
\cdot\frac{x_m}{x_{m_j}}}{[g_jx_m, g_ix_m]}\\
=& \frac{g_ig_jx_{m_{ij}}\cdot\frac{x_m}{x_{m_i}}
\cdot\frac{x_m}{x_{m_j}}}{x_m[g_j, g_i]}
=\frac{g_ig_jx_{m_{ij}}\cdot\frac{x_m}{x_{m_i}}
\cdot\frac{x_m}{x_{m_j}}}{x_mg_ig_j}
=\frac{x_{m_{ij}}x_m}{x_{m_i}x_{m_j}}=1
\end{align*}
as required. So Lemma \ref{lemma 2.2} is proved.
\end{proof}
		
\begin{lemma}\label{lemma'}
Let $A,B$ and $C$ be positive integers with $C|A$. If $[\frac{A}{C},B]=A$, then $C|B$.
\end{lemma}
\begin{proof}
Since $C|A$, we can let $A=CK$ with $K\in \mathbb Z^+$. Hence
$$\Big[\frac{A}{C},B\Big]=[K,B]=\frac{KB}{(K,B)}=CK$$
and so $B=C(K,B)$. Therefore $C|B$. Lemma
\ref{lemma'} is proved.
\end{proof}

\begin{lemma}{\rm\cite{[ZY-2025]}}\label{lemma ''}
Let $A_1,\cdots, A_r$ and $B$ be pairwise distinct
positive integers. Then
\begin{align*}
[(A_1,\cdots,A_r),B]=([A_1,B],\cdots,[A_r,B]).
\end{align*}
\end{lemma}

Now for any $x,z\in S$ with $z|x$ and $z<x$, we
define the set $M_S(z,x)$ by
$$M_S(z,x):=\{u\in S: z|u|x\}.$$
Then $z,x\in M_S(z,x)$.
We have the following result.
\begin{lemma}\label{lemma '''}
Let $S$ be a gcd-closed set satisfying the condition $\mathcal G$ and let $x\in S$ satisfy
$|G_S(x)|\ge 2$ and $y\in G_S(x)$. Let $z\in S$ be such that
$z|x, z<x$ and $z\nmid y$. If the set $\{z_e, z_{e-1},\cdots, z_1, z_0\}$ is a subset of $M_S(z,x)$ and satisfies
that $z_i\in G_S(z_{i-1})$ for all $1\le i\le e$, where $z_e:=z$, $z_0:=x$,
then for any integer $i$ with $0\le i\le e$, we have $[y,z_i]=x$.
\end{lemma}
\begin{proof}
Set $d_0:=y$ and $d_i:=(z_i, d_{i-1})$
for $1\le i\le e$. It follows from the hypothesis that $x$
satisfies the condition
$\mathcal G$ and $\{z_1, y\}\subseteq G_S(x)$ that
\begin{align}\label{e3}
d_1=(z_1, y)\in G_S(z_1)\cap G_S(y).
\end{align}
Now we assert that
\begin{align}\label{e6}
d_i\in G_S(z_i)\cap G_S(d_{i-1})\ \hbox{for}\ 1\le i\le e.
\end{align}

If $i=1$, then by \eqref{e3} we know that \eqref{e6} is true.

If $i=2$, then from \eqref{e3} and the hypothesis that $z_2\in G_S(z_1)$ we derive
that $\{z_2, d_1\}\subseteq G_S(z_1)$. The condition $z\nmid y$ guarantees that $z_2\ne d_1$.
Otherwise, $z_2=(z_1,d_0)=(z_1,y)$ which means that $z\mid z_2\mid y$. This is impossible.
It follows that $d_2=(z_2,d_1)\in G_S(z_2)\cap G_S(d_1)$
since $z_1$ satisfies the condition
$\mathcal G$. Thus \eqref{e6} is true when $i=2$.

Assume that \eqref{e6} holds for the $i$ case, where $1\le i\le e-1$. Now we consider
the $i+1$ case. It follows from \eqref{e6} and $z_i\in G_S(z_{i-1})$ for $1\le i\le e$
that $\{z_{i+1}, d_i\}\subseteq G_S(z_i)$. Notice that $z\nmid y$ guarantees that
$d_i\ne z_{i+1}$ for $0\le i\le e-1$. Otherwise, $z_{i+1}=d_i=(z_i,d_{i-1})$ which
implies that $z_{i+1}\mid d_{i-1}$. But $d_{i-1}\mid d_{i-2}\mid \cdots \mid d_1\mid d_0=y$
and $z=z_e\mid \cdots \mid z_{i+1}$. Hence we deduce that $z\mid z_{i+1}\mid d_{i-1}\mid y$
which infers that $z\mid y$. This is a contradiction. So we must have
$d_i\ne z_{i+1}$ for $0\le i\le e-1$. Now let us continue to complete the proof of
assertion \eqref{e6}. Since $z_i$ satisfies the condition $\mathcal G$, we have
$d_{i+1}=(z_{i+1}, d_i)\in G_S(z_{i+1})\cap G_S(d_i)$ for $0\le i\le e-1$. So
\eqref{e6} is true for the $i+1$ case. Thus claim \eqref{e6} is proved.

Now for $0\le i\le e-1$, since $z_i$ satisfies
the condition $\mathcal G$ and $z_{i+1}\in G_S(z_i)$ for $1\le i\le e-1$,
from \eqref{e6}, we deduce that
\begin{align}\label{e7}
[d_{i}, z_{i+1}]=z_{i}.
\end{align}
On the other hand, from \eqref{e6}, one deduces that
\begin{align}\label{e8}
d_e|d_{e-1}|\cdots |d_1|y.
\end{align}

Consequently, for $0\le i\le e-1$, by \eqref{e7}
and \eqref{e8}, we have
$$[y,z_i]=[y,[d_{i}, z_{i+1}]]=[[y,d_{i}], z_{i+1}]=[y, z_{i+1}].$$
It then follows that
$$x=[y,x]=[y,z_0]=[y,z_1]=\cdots =[y,z_{e-1}]=[y,z_e]=[y,z]$$
as desired.

This completes the proof of Lemma \ref{lemma '''}.
\end{proof}

\begin{lemma}\label{lemma 2.5}
Let $S$ be a gcd-closed set satisfying the condition
$\mathcal{G}$ and let $x_m,x_l \in S$ with $G_S(x_m)=\{x_{m_1},\cdots, x_{m_s}\}$.
Let $1\le k\le s-1$. If $(x_l,x_m)|x_{m_{1\cdots k}}$
and $(x_l,x_m)\nmid x_{m_j} $ for all integers $j$ with
$k+1\le j\le s$, then the following statements hold:

{\rm (i). } $P_{k+1}\cdots P_s \mid (x_l,x_m)$.

{\rm (ii). } Let $\beta:=\frac{(x_l,x_m)}{P_{k+1}\cdots P_s}$.
Then $\beta\mid x_{m_{1\cdots s}}$ and
$\Big(P_{k+1}\cdots P_s,\dfrac{x_{m_{1\cdots s}}}{\beta}\Big)=1$.
\end{lemma}
\begin{proof} We divide the proof into the following two cases.

Case 1. $(x_l,x_m) = x_{m_{1\cdots k}}$. Then by the assumption,
we have $x_{m_{1\cdots k}}\nmid x_{m_j} $
for all integers $j$ with $k+1\le j\le s$. By Lemma \ref{lemma 2.3},
we deduce that
$$x_{m_{1\cdots k}}P_1\cdots P_k=x_m=x_{m_{1\cdots s}}P_1\cdots P_s.$$
Thus
$$(x_l,x_m) = x_{m_{1\cdots k}}=\frac{x_{m_{1\cdots s}}
P_1\cdots P_kP_{k+1}\cdots P_s}{P_1\cdots P_k}
=x_{m_{1\cdots s}}P_{k+1}\cdots P_s.$$
Hence $P_{k+1}\cdots P_s \mid (x_l,x_m)$ as expected. At this
moment, one has $x_{m_{1\cdots s}}=\beta$. It then follows that
$$\Big(P_{k+1}\cdots P_s,\frac{x_{m_{1\cdots s}}}{\beta}\Big)
=(P_{k+1}\cdots P_s,1)=1$$
as desired. Parts (i) and (ii) are proved in this case.
		
Case 2. $(x_l,x_m) \mid x_{m_{1\cdots k}}$,
$(x_l,x_m)<x_{m_{1\cdots k}}$ and
$(x_l,x_m) \nmid x_{m_j} $ for all $k+1\le j\le s$.
Then $(x_l,x_m) \nmid x_{m_{1\cdots kj}}$
for all $k+1\le j\le s$. Now we define a set
$M$ as follows. If $(x_l,x_m)\in G_S(x_{m_{1\cdots k}})$,
then we define
$M:=\{(x_l,x_m), x_{m_{1\cdots k}}\}.$
If $(x_l,x_m)\notin G_S(x_{m_{1\cdots k}})$, then
we can pick an $x_{l_1}\in S$ which satisfies that
$(x_l,x_m)|x_{l_1}$ and $x_{l_1}\in G_S(x_{m_{1\cdots k}})$.
If $(x_l,x_m)\in G_S(x_{l_1})$, we define
$M:=\{(x_l,x_m), x_{l_1}, x_{m_{1\cdots k}}\}.$
If $(x_l,x_m)\notin G_S(x_{l_1})$, then we can continue
this process, and after a finite number of steps, says
$e-1$ steps $(e\ge 1)$, we obtain the set $M$ as follows:
$M:=\{x_{l_e}, x_{l_{e-1}}, \cdots , x_{l_1 }, x_{l_0}\},$
where
\begin{align}\label{II}
x_{l_e}:=(x_l,x_m),
\end{align}
\begin{align}\label{III}
x_{l_0}:=x_{m_{1\cdots k}},
\end{align}
and
$x_{l_i } \in G_S(x_{l_{i-1}})$ for $1\le i\le e$.

Since $M$ is a subset of $M_S((x_l,x_m), x_{m_{1\cdots k}})$
and satisfies that $x_{l_i }\in G_S(x_{l_{i-1}})$ for all
$1\le i\le e$, by Lemma \ref{lemma1} one has $x_{m_{1\cdots kj}}
\in G_S(x_{m_{1\cdots k}})$, and by the hypothesis, we deduce that
$(x_l,x_m) \nmid x_{m_{1\cdots kj}}$ for all $k+1\le j\le s$. So
Lemma \ref{lemma '''} applied to the set $M$ gives that
\begin{align}\label{eq2.4}
[x_{m_{1\cdots kj}}, x_{l_i}]=x_{m_{1\cdots k}}
\end{align}
for $0\le i\le e$ and $k+1\le j\le s$. Since $M$ is a divisor chain,
it follows from Lemma \ref{lemma ''} and \eqref{eq2.4} that
\begin{align}\label{eq2.3}
&[(x_{m_{1\cdots kj}}, x_{l_{i-1}}), x_{l_{i}}]
=([x_{m_{1\cdots kj}},x_{l_{i}}],[x_{l_{i-1}}, x_{l_{i}}])
=(x_{m_{1\cdots k}},x_{l_{i-1}})=x_{l_{i-1}}
\end{align}
for $1\le i\le e$ and $k+1\le j\le s$.

For $0\le i\le e$, set
\begin{align}\label{eq2.10}
w_i:=(x_{m_{1\cdots s}}, x_{l_i}).
\end{align}
Then $w_i\mid x_{m_{1\cdots s}}\mid x_{m_{1\cdots kj}}$ and $w_i\mid x_{l_i}\mid x_{l_{i-1}}$,
and so $w_i\mid (x_{m_{1\cdots kj}},x_{l_{i-1}})$. So by \eqref{eq2.3}, we have
\begin{align}\label{e2.23}
\Big[\frac{(x_{m_{1\cdots kj}}, x_{l_{i-1}})}{w_{i}}, \frac{x_{l_{i}}}{w_{i}}\Big]
=\frac{x_{l_{i-1}}}{w_{i}}
\end{align}
for $1\le i\le e$ and $k+1\le j\le s$. However, by \eqref{eq2.4}
and Lemma \ref{lemma 2.3} we get that
\begin{align}\label{e2.24}
&\frac{(x_{m_{1\cdots kj}},x_{l_{i-1}})}{w_{i}}
=\frac{x_{m_{1\cdots kj}}x_{l_{i-1}}}
{[x_{m_{1\cdots kj}}, x_{l_{i-1}}]w_{i}}=\frac{x_{m_{1\cdots kj}}x_{l_{i-1}}}
{x_{m_{1\cdots k}}w_{i}}
=\frac{\frac{x_{m_{1\cdots s}}P_1\cdots
\cdots P_sx_{l_{i-1}}}{P_1\cdots P_kP_j}}
{\frac{x_{m_{1\cdots s}}P_1\cdots P_sw_{i}}{P_1\cdots P_k}}
=\frac{x_{l_{i-1}}}{w_{i}P_j}.
\end{align}
This implies that $P_j\mid \frac{x_{l_{i-1}}}{w_{i}}$.
Putting \eqref{e2.24} into \eqref{e2.23} gives that
$$\Big[\frac{x_{l_{i-1}}}{w_{i}P_j}, \frac{x_{l_{i}}}{w_{i}}\Big]=\frac{x_{l_{i-1}}}{w_{i}}.$$
So by Lemma \ref{lemma'}, we have
\begin{align}\label{eq''}
P_j\mid \frac{x_{l_{i}}}{w_{i}}\mid x_{l_{i}}
\end{align}
for $1\le i\le e$ and $k+1\le j\le s$. By Lemma \ref{lemma 2.2}, we know that
$P_{k+1},\cdots,P_s$ are pairwise relatively prime. It then follows that
\begin{align}\label{eq2.9}
P_{k+1}\cdots P_s \mid x_{l_{i}}\ \hbox{for}\ 1\le i\le e.
\end{align}
Particularly, $P_{k+1}\cdots P_s\mid x_{l_e}=(x_l,x_m)$ as desired.
Part (i) is proved in this case.

Let us now show part (ii) for Case 2. It follows
from \eqref{III} and Lemma \ref{lemma 2.3} that
\begin{align}\label{I}
x_{l_0}=x_{m_{1\cdots k}}=x_{m_{1\cdots s}}P_{k+1}\cdots P_s.
\end{align}
This together with \eqref{eq''} yields that $P_{k+1}\cdots P_s \mid x_{l_{i}}$
for all integers $i$ with $0\le i\le e$. So for any $0\le i\le e$, we may define
the integer $c_i$ by
\begin{align}\label{eq2.15}
c_i:=\frac{x_{l_i}}{P_{k+1}\cdots P_s}.
\end{align}
Then by \eqref{I}, we have
\begin{align}\label{B}
c_0=\frac{x_{l_0}}{P_{k+1}\cdots P_s}=\frac{x_{m_{1\cdots s}}P_{k+1}\cdots P_s}{P_{k+1}\cdots P_s}
=x_{m_{1\cdots s}}.
\end{align}
It is clear from \eqref{II} that
\begin{align}\label{BB}
c_e=\frac{x_{l_e}}{P_{k+1}\cdots P_s}=\frac{(x_l,x_m)}{P_{k+1}\cdots P_s}=\beta.
\end{align}
Since $x_{l_{i} } \in G_S(x_{l_{i-1}})$
for $1\le i\le e$, by \eqref{eq2.15} one has $c_i|c_{i-1}$ and $c_i<c_{i-1}$ for $1\le i\le e$.
Furthermore, by \eqref{eq2.10} and \eqref{eq2.15}, we can deduce that
\begin{align}\label{eq2.16}
w_i&=(x_{l_{i}},c_0)
=(c_iP_{k+1}\cdots P_s, c_0)=c_i\Big(P_{k+1}\cdots P_s,\frac{c_0}{c_i}\Big)
\end{align}
for any integer $i$ with $1\le i\le e$. It follows from \eqref{eq2.15} and \eqref{eq2.16} that
\begin{align}\label{eq2.17}
\dfrac{x_{l_i}}{w_i}=&\dfrac{c_iP_{k+1}\cdots P_s}{c_i(P_{k+1}\cdots P_s,\frac{c_0}{c_i})}
=\dfrac{P_{k+1}\cdots P_s}{\frac{P_{k+1}\cdots P_s\frac{c_0}{c_i}}{[P_{k+1}\cdots P_s,\frac{c_0}{c_i}]}}
=\frac{[P_{k+1}\cdots P_s,\frac{c_0}{c_i}]}{\frac{c_0}{c_i}}.
\end{align}

On the other hand, by \eqref{eq''} and Lemma \ref{lemma 2.2}, we know that
\begin{equation}\label{eq2.29}
P_{k+1}P_{k+2}\cdots P_s\mid \frac{x_{l_{i}}}{w_{i}}
\end{equation}
for $1\le i\le e$. Let
$$[P_{k+1}P_{k+2} \cdots P_s,\frac{c_0}{c_i}]:=v_iP_{k+1}P_{k+2} \cdots P_s$$
with $v_i \in \mathbb{Z}^+$.
This together with \eqref{eq2.17} and \eqref{eq2.29} yields that
\begin{equation*}
P_{k+1}\cdots P_s \mid \frac{v_iP_{k+1}\cdots P_s}
{\frac{c_0}{\mathcal c_i}}
\end{equation*}
which implies that $\dfrac{c_0}{\mathcal c_i} \mid v_i$.
Now we let $v_i=\dfrac{c_0}{c_i}u_i$ for some
$u_i\in \mathbb{Z}^+$. Then we can derive that
\begin{align*}
&\frac{P_{k+1}\cdots P_s\frac{c_0}{c_i}}
{(P_{k+1}\cdots P_s,\frac{c_0}{c_i})}
=\Big[P_{k+1}\cdots P_s,\frac{c_0}{c_i}\Big]
=v_iP_{k+1}\cdots P_s
= \frac{c_0}{c_i}u_iP_{k+1}\cdots P_s.
\end{align*}
Then for $1\le i\le e$, one has
$$(P_{k+1}\cdots P_s,\frac{c_0}{c_i})u_i=1$$
and so
$$(P_{k+1}\cdots P_s,\frac{c_0}{c_i})=u_i=1.$$
In particular, by \eqref{B} and \eqref{BB} we have
$$\Big(P_{k+1}\cdots P_s,\frac{c_0}{c_e}\Big)
=\Big(P_{k+1}\cdots P_s,\frac{x_{m_{1\cdots s}}}{\beta}\Big)=1$$
as required.
This concludes the proof of part (ii) and that of Lemma \ref{lemma 2.5}.
\end{proof}
		
\begin{lemma}\label{lemma 2.6}
Let $S$ be a gcd-closed set satisfying the condition $\mathcal{G}$
and $x_l, x_m\in S$ with $G_S(x_m)=\{x_{m_1},\cdots,x_{m_s}\}$.
Let $1\le k\le s-1$. Assume that $(x_l,x_m)\mid x_{m_{1\cdots k}}$
and $(x_l,x_m) \nmid x_{m_j}$ for all $k+1 \le j \le s$.
Let $1 \le i_1<\cdots < i_h \le k$ with $1\le h\le k$
and $k+1 \le j_1<\cdots<j_t \le s$ with $1\le t\le s-k$.
Then each of the following is true.	
		
{\rm (i).} $(x_l, x_{m_{i_1\cdots i_h}})=(x_l, x_m).$

{\rm (ii).} $[x_l,x_{m_{j_1\cdots j_t}}]=[x_l,x_m].$

{\rm (iii).}
$[x_{l}, x_{m_{i_1\cdots i_hj_1\cdots j_t}}]
=[x_l,x_{m_{i_1\cdots i_h}}].$

{\rm (iv).} $(x_{l}, x_{m_{i_1\cdots i_hj_1\cdots j_t}})
=(x_l, x_{m_{j_1\cdots j_t}})$.
			
\end{lemma}
\begin{proof}
(i). Since $(x_l, x_m)\mid x_{m_{1\cdots k}}$ and $(x_l,x_m)\mid x_l$,
we have $(x_l,x_m)\mid(x_l,x_{m_{1\cdots k}})$. Since
$1\le i_1<\cdots<i_h\le k$, one has
$x_{m_{1\cdots k}}\mid x_{m_{i_1\cdots i_h}}\mid x_m$.
It then follows that
$(x_l,x_{m_{1\cdots k}})\mid(x_l,x_{m_{i_1\cdots i_h}})\mid(x_l,x_m),$
and so we must have
\begin{align*}
(x_l,x_m)=(x_l,x_{m_{1\cdots k}})=(x_l,x_{m_{i_1\cdots i_h}})
\end{align*}
as expected. Part (i) is proved.

(ii). Since $(x_l,x_m)\mid x_{m_{1\cdots k}}$ and
$(x_l,x_m) \nmid x_{m_j} $ for all $j$ with $k+1\le j\le s$,
by Lemma \ref{lemma 2.5} (i), we have
$P_{k+1}\cdots P_{s}\mid (x_l,x_m)$.
One may let
\begin{align}\label{1}
(x_l,x_m):=\beta P_{k+1}\cdots P_{s}.
\end{align}
From Lemma \ref{lemma 2.5} (ii), we obtain that
$\beta|x_{m_{12\cdots s}}$
and
\begin{equation}\label{eq44.18}
\Big(P_{k+1}\cdots P_{s}, \dfrac{x_{m_{1\cdots s}}}{\beta}\Big)=1.
\end{equation}

It is easy to see that for $k+1\le j_1<\cdots<j_t\le s$,
we have
\begin{align*}
x_{m_{(k+1)\cdots s}}\mid x_{m_{j_1\cdots j_t}} \mid x_{m}.
\end{align*}
So with Lemma \ref{lemma 2.3} applied to
$x_{m_{(k+1)\cdots s}}$, one gets that
\begin{equation}\label{eq4.10}
x_{m_{(k+1)\cdots s}}=\frac{x_{m_{1\cdots s}}P_1\cdots P_s}{P_{k+1}\cdots P_s}
=x_{m_{1\cdots s}}P_1\cdots P_k \mid x_{m_{j_1\cdots j_t}}
\mid x_{m_{1\cdots s}}P_1\cdots P_s=x_m.
\end{equation}	
From Lemma \ref{lemma 2.2}, we can deduce that $(P_{k+1}\cdots P_{s},P_1\cdots P_k)=1$.
Then by \eqref{eq44.18}, one has
\begin{align*}
(P_{k+1}\cdots P_s, \dfrac{x_{m_{1\cdots s}}}{\beta}\cdot P_1\cdots P_k)=1.
\end{align*}
It then follows that
\begin{equation}\label{eq44.10}
[P_{k+1}\cdots P_s, \dfrac{x_{m_{1\cdots s}}}{\beta}\cdot P_1\cdots P_k]
=\dfrac{x_{m_{1\cdots s}}P_1\cdots P_s}{\beta}.
\end{equation}
Thus by \eqref{1}, \eqref{eq4.10}, \eqref{eq44.10} and
Lemma \ref{lemma 2.3}, we can deduce that
\begin{align}
[(x_l,x_m),x_{m_{j_1\cdots j_t}}]
=&[\beta P_{k+1}\cdots P_s, x_{m_{j_1\cdots j_t}}]\notag\\
=&[\beta P_{k+1}\cdots P_s,[x_{m_{1\cdots s}}
P_1\cdots P_k , x_{m_{j_1\cdots j_t}}]]\notag\\
=&\Big[\beta\Big[P_{k+1}\cdots P_s, \frac{x_{m_{1\cdots s}}}
{\beta}P_1\cdots P_k\Big], x_{m_{j_1\cdots j_t}}\Big]\notag\\
=&\Big[\beta\dfrac{x_{m_{1\cdots s}}P_1\cdots P_s}{\beta},
x_{m_{j_1\cdots j_t}}\Big]\notag\\
=&[x_{m_{1\cdots s}}P_1\cdots P_s,x_{m_{j_1\cdots j_t}}]\notag\\
=&x_{m_{1\cdots s}}P_1\cdots P_s
=x_m.\label{IV}
\end{align}
Hence by \eqref{IV}, we have
$$[x_l,x_{m_{j_1\cdots j_t}}]=[[x_l,(x_l,x_m)],x_{m_{j_1\cdots j_t}}]
=[x_l,[(x_l,x_m),x_{m_{j_1\cdots j_t}}]]=[x_l,x_m ]$$
as one desires. Part (ii) is proved.
			
(iii). Since $1\le i_1<\cdots<i_h\le k$
and $k+1 \le j_1<\cdots<j_t\le s$, we have
\begin{equation*}
x_{m_{i_1\cdots i_h(k+1)\cdots s}}
\mid x_{m_{i_1\cdots i_hj_1\cdots j_t}}
\mid x_{m_{i_1\cdots i_h}}.
\end{equation*}
Then by Lemma \ref{lemma 2.3}, one has
\begin{equation}\label{eq4.8}
x_{m_{i_1...i_h (k+1)...s}}
=\frac{x_{m_{1\cdots s}}P_1\cdots P_s}
{P_{i_1}\cdots P_{i_h}P_{k+1}\cdots P_s}
=\frac{x_{m_{1\cdots s}}P_1\cdots P_k}
{P_{i_1}\cdots P_{i_h}}\Big|x_{m_{i_1\cdots i_h j_1\cdots j_t}}
\Big|\frac{x_{m_{1\cdots s}}P_1\cdots  P_s}
{P_{i_1}\cdots P_{i_h}}.
\end{equation}		
But Lemma \ref{lemma 2.2} tells us that
\begin{align}\label{V}
\Big(P_{k+1}\cdots P_{s},\dfrac{P_1\cdots P_k}
{P_{i_1}\cdots P_{i_h}}\Big)=1.
\end{align}
Hence by \eqref{eq44.18} and \eqref{V}, we have
$$\Big(P_{k+1}\cdots P_s, \frac{x_{m_{1\cdots s}}}
{\beta}\cdot\dfrac{P_1\cdots P_k}{P_{i_1}\cdots P_{i_h}}\Big)=1. $$
It follows that
\begin{equation}\label{eq44.9}
\Big[P_{k+1}\cdots P_s, \frac{x_{m_{1\cdots s}}}{\beta}\cdot
\dfrac{P_1\cdots P_k}{P_{i_1}\cdots P_{i_h}}\Big]=
\dfrac{x_{m_{1\cdots s}}P_1\cdots P_s}{\beta P_{i_1}\cdots P_{i_h}}.
\end{equation}
So by \eqref{1}, \eqref{eq4.8}, \eqref{eq44.9} and
Lemma \ref{lemma 2.3}, one has
\begin{align}
[(x_l,x_m), x_{m_{i_1\cdots i_hj_1\cdots j_t}}]=&[\beta P_{k+1}
\cdots P_s, x_{m_{i_1\cdots i_hj_1\cdots j_t}}]\notag\\
=&\Big[\beta P_{k+1}\cdots P_s, \Big[\dfrac{x_{m_{1\cdots s}}P_1\cdots P_k}
{P_{i_1}\cdots P_{i_h}},x_{m_{i_1\cdots i_hj_1\cdots j_t}}\Big]\Big] \notag\\
=&\Big[\beta\Big[P_{k+1}\cdots P_s, \dfrac{x_{m_{1\cdots s}}P_1\cdots P_k}
{\beta P_{i_1}\cdots P_{i_h}}\Big],x_{m_{i_1\cdots i_hj_1\cdots j_t}}\Big] \notag\\
=&\Big[\beta\dfrac{x_{m_{1\cdots s}}P_1\cdots P_s}{\beta P_{i_1}\cdots P_{i_h}},
x_{m_{i_1\cdots i_hj_1\cdots j_t}}\Big]\notag\\
=&\Big[\dfrac{x_{m_{1\cdots s}}P_1\cdots P_s}{P_{i_1}\cdots P_{i_h}},
x_{m_{i_1\cdots i_hj_1\cdots j_t}}\Big]\notag\\
=&\dfrac{x_{m_{1\cdots s}}P_1P_2\cdots  P_s}{P_{i_1}\cdots P_{i_h}}\notag\\
=&x_{m_{i_1\cdots i_h}}. \notag
\end{align}
It follows that
\begin{align*}
&[x_l,x_{m_{i_1\cdots i_hj_1\cdots j_t}}]
=[[x_l,(x_l,x_m)],x_{m_{i_1\cdots i_hj_1\cdots j_t}}]\\
=&[x_l,[(x_l,x_m),x_{m_{i_1\cdots i_hj_1\cdots j_t}}]]
=[x_l,x_{m_{i_1\cdots i_h}}]
\end{align*}
as expected. Part (iii) is proved.

(iv). By parts (i) to (iii), and applying again Lemma
\ref{lemma 2.3}, one arrives at
\begin{align*}
(x_l,x_{m_{i_1\cdots i_hj_1\cdots j_t}})
=&\frac{x_lx_{m_{i_1\cdots i_hj_1\cdots j_t}}}
{[x_l,x_{m_{i_1\cdots i_hj_1\cdots j_t}}]}
=\frac{x_lx_{m_{i_1\cdots i_hj_1\cdots j_t}}}
{[x_l,x_{m_{i_1\cdots i_h}}]}\\
=&\frac{x_lx_{m_{i_1\cdots i_hj_1\cdots j_t}}}
{\frac{x_lx_{m_{i_1\cdots i_h}}}{(x_l,x_{m_{i_1\cdots i_h}})}}
=\frac{x_{m_{i_1\cdots i_hj_1\cdots j_t}}
(x_l,x_{m_{i_1\cdots i_h}})}{x_{m_{i_1\cdots i_h}}}\\
=& \frac{x_{m_{i_1\cdots i_hj_1\cdots j_t}}(x_l,x_m)}
{x_{m_{i_1\cdots i_h}}}
=\frac{x_{m_{i_1\cdots i_hj_1\cdots j_t}}x_lx_m}
{x_{m_{i_1\cdots i_h}}[x_l,x_m]}\\
=& \frac{x_{m_{i_1\cdots i_hj_1\cdots j_t}}x_lx_m}
{x_{m_{i_1\cdots i_h}}[x_l,x_{m_{j_1\cdots j_t}}]}
=\frac{\frac{x_m}{P_{i_1}\cdots P_{i_h}P_{j_1}
\cdots P_{j_t}}\cdot x_lx_m}{\frac{x_m}
{P_{i_1}\cdots P_{i_h}}\cdot[x_l,x_{m_{j_1\cdots j_t}}]}\\
=& \frac{x_l\cdot\frac{x_m}{P_{j_1}\cdots P_{j_t}}}
{[x_l,x_{m_{j_1\cdots j_t}}]}
=\frac{x_lx_{m_{j_1\cdots j_t}}}
{[x_l,x_{m_{j_1\cdots j_t}}]}
=(x_l,x_{m_{j_1\cdots j_t}})
\end{align*}
as required. Part (iv) is proved.

The proof of Lemma \ref{lemma 2.6} is complete.			
\end{proof}	

\section{Reductions of $c_{ij}$ and $\alpha_{\xi_a}(x_m)$}

\begin{lemma}\label{lemma 2.7}
If $S$ is a gcd-closed set of $n$ positive integers, then the power GCD matrix $(S^a)$
is nonsingular and
\begin{align*}
((S^a)^{-1})_{ij}:=\sum_{x_i|x_k\atop x_j|x_k}
\frac{c_{i k}c_{j k}}{\alpha _{\xi_a}(x_k)}
\end{align*}
with
\begin{equation}\label{eq3.1}
c_{ij}:=\sum _{dx_i|x_j\atop dx_i\nmid x_t, x_t<x_j}\mu(d)
\end{equation}
and $\alpha_{\xi_a}(x_k)$ is defined as in \eqref{eq 1.1}.
\end{lemma}
\begin{proof} From \cite[Example 1 (ii)]{[BL-LMA1993]},
one knows that the power GCD matrix $(S^a)$
is positive definite, and so is nonsingular. Setting
$f=\xi_a$ in \cite[Theorem 3]{[BL-LMA1993]} gives the expected result.
Lemma \ref{lemma 2.7} is proved.
\end{proof}

The following result is due to Hong gotten in 2004
which reduces greatly the formula of Bourque and Ligh.

\begin{lemma}{\rm\cite[Theorem 1.2]{[H-JA2004]}}\label{lemma 2.8}
Let $S$ be a gcd-closed set and let $f$ be a complex-valued function.
For any $x\in S$, define the function $\alpha_f$ as follows:
$$\alpha_f(x):=\sum\limits_{d|x\atop d\notin \varUpsilon_S(x)}(f*\mu)(d)$$
with $\varUpsilon_S(x):=\{z\in \mathbb Z^+: \exists \ y\in S, y<x, z|y\}$.
Then $\alpha_f(x)=\beta_f(x)$, where
$$
\beta_f(x):=\sum\limits_{J\subseteq G_S(x)}(-1)^{|J|}f(\gcd(J\cup\{x\})).
$$
\end{lemma}

The arithmetic function $\delta$ is defined for any positive integer $x$ by
$$
\delta(x):=\left\{
\begin{aligned}
&1\ \hbox{if}\ x=1,\\
&0\ \hbox{if}\ x>1.
\end{aligned}
\right.
$$
For any integers $i$ and $j$ with $1\le i,j\le n$, we define
two sets $S_{ij}$ and $G_{S}^{(i)}(x_j)$ as follows:
$$S_{ij}:=\Big\{\frac{x_t}{x_i}: x_t\in S, x_i|x_t, x_t\le x_j\Big\}$$
and
$$G_{S}^{(i)}(x_j):=\{y\in G_{S}(x_j): x_i \mid y\}.$$
It is easy to check that if $S$ is gcd closed, then $S_{ij}$ is gcd closed
when $x_i\le x_j$ and $S_{ij}$ is empty when $x_i>x_j$.
\begin{lemma}\label{lemma 2.9}
Let $S$ be a gcd-closed set and let $c_{ij}$ be defined
as in {\rm(\ref{eq3.1})}.

{\rm (i).} If $x_i\nmid x_j$, then $c_{ij} = 0$.
	
{\rm (ii).} If $x_i=x_j$, then $c_{ij}=1$.
	
{\rm (iii).} If $x_i \mid x_j$, $x_i \neq x_j$,
then
$$c_{ij}=\sum\limits_{J\subseteq G_{S_{ij}}\big(\frac{x_j}{x_i}\big)}(-1)^{|J|}
\delta\Big(\gcd\Big(J\cup\Big\{\frac{x_j}{x_i}\Big\}\Big)\Big)
=\sum_{\bar{J}\subseteq G_{S}^{(i)}(x_j)}(-1)^{|\bar{J}|}
\delta\Big(\frac{\gcd(\bar{J}\cup \{x_j\})}{x_i}\Big).$$
\end{lemma}

\begin{proof}
{\rm (i).}  From the definition of $c_{ij}$ as in
\eqref{eq3.1}, we deduce that if $x_i \nmid x_j$, then $c_{ij} = 0$.
Part (i) is proved.
	
{\rm (ii).}  If $x_i=x_j$, we can easily deduce that
$$c_{jj}=\sum\limits_{dx_j \mid x_j \atop dx_j
\nmid x_t, x_t < x_j}\mu(d)=\mu(1)=1$$
as required. Part (ii) is proved.

{\rm (iii).} If $x_i \mid x_j$ and $x_i \neq x_j$,
we can deduce from \eqref{eq3.1} that
\begin{equation}\label{eq3.12}
c_{ij}=\sum\limits_{dx_i \mid x_j,dx_i \nmid x_t
\atop x_t < x_j}\mu(d) =\sum\limits_{d \mid \frac{x_j}{x_i},
d \nmid \frac{x_t}{x_i} \atop x_i \mid x_t, x_t < x_j}\mu(d).
\end{equation}
In what follows, we show that $S_{ij}$ is a gcd-closed set.
For any $\frac{x_{t_1}}{x_i}, \frac{x_{t_2}}{x_i}\in S_{ij}$,
we have $x_{t_1},x_{t_2}\in S$, $x_i\mid x_{t_1}, x_i\mid x_{t_2}$,
$x_{t_1}\le x_j$ and $x_{t_2}\le x_j$. Then $x_i\mid (x_{t_1},x_{t_2})$
and $(x_{t_1},x_{t_2})\le x_{t_1}\le x_j$. Since $S$ is gcd closed, $(x_{t_1},x_{t_2})\in S$.
It infers that
$$\Big(\frac{x_{t_1}}{x_i}, \frac{x_{t_2}}{x_i}\Big)=\frac{(x_{t_1},x_{t_2})}{x_i}\in S_{ij}.$$
Hence $S_{ij}$ is gcd closed. Then by \eqref{eq3.12}, Lemma \ref{lemma 2.8}
and the definition of $S_{ij}$, we deduce that
\begin{align}
c_{ij}=&\sum _{\begin{tiny}\begin{array}{c}d|\frac{x_j}{x_i}, d\nmid \frac{x_t}{x_i}\notag\\
\frac{x_t}{x_i}\in S_{ij}, x_t<x_j \end{array}\end{tiny}}\mu(d)\notag\\
=&\sum _{\begin{tiny}\begin{array}{c}d|\frac{x_j}{x_i}, d\nmid \frac{x_t}{x_i}\notag\\
\frac{x_t}{x_i}\in S_{ij}, x_t<x_j \end{array}\end{tiny}}\delta*\mu(d)\notag\\
=&\sum _{\begin{tiny}\begin{array}{c}d|\frac{x_j}{x_i}, d\nmid \frac{x_t}{x_i}\notag\\
\frac{x_t}{x_i}\in S_{ij}, \frac{x_t}{x_i}<\frac{x_j}{x_i}\end{array}\end{tiny}}\delta*\mu(d)\notag\\
=&\alpha_{\delta}\Big(\frac{x_j}{x_i}\Big)\notag\\=&\sum_{J\subseteq G_{S_{ij}}(\frac{x_j}{x_i})}(-1)^{|J|}
\delta(\gcd\Big(J\cup \Big\{\frac{x_j}{x_i}\Big\}\Big)).\label{i}
\end{align}
The first formula is proved.

Now we show the second one. We can easily see that if $x_i\mid x_j$, then
\begin{align*}
G_{S_{ij}}\Big(\frac{x_j}{x_i}\Big)=\frac{1}{x_i}G_{S}^{(i)}(x_j).
\end{align*}
So for any subset $J\subseteq G_{S_{ij}}\big(\frac{x_j}{x_i}\big)$, one can write
$J=\frac{1}{x_i}\bar{J}$ with $\bar{J}\subseteq G_{S}^{(i)}(x_j)$.
It then follows from \eqref{i} that
\begin{align*}
c_{ij}=\sum_{\bar{J}\subseteq G_{S}^{(i)}(x_j)}(-1)^{|\bar{J}|}\delta\Big(\frac{\gcd(\bar{J}\cup \{x_j\})}{x_i}\Big)
\end{align*}
as required.

This finishes the proof of Lemma \ref{lemma 2.9}.
\end{proof}

\begin{lemma}\label{lemma 2.10}
Let $S$ be a gcd-closed set satisfying the condition
$\mathcal G$ and $x_m \in S$. Then for any integer $r$
with $1\le r\le m$, one has
\begin{align*}
c_{rm}=\left\{\begin{array}{cl}
1 &\hbox{if}\ r=m,\\
(-1)^k &\hbox{if} \ r=m_{i_1\cdots i_k} \ \hbox{for \ some} \
1\leq i_1<\cdots<i_k\leq s \ \hbox{and} \ 1 \leq k \leq s,\\
0 &\hbox{otherwise}.
\end{array}\right.
\end{align*}
\end{lemma}
\begin{proof}
From the definition of $c_{ij}$ as in $(\ref{eq3.1})$,
one knows that if $x_r \nmid x_m$, then $c_{rm} = 0$.
In what follows we let $x_r \mid x_m$.
	
Let $r=m$. Then we can easily deduce that
$$c_{mm}=\sum\limits_{dx_m|x_m}\mu(d)=\mu(1)=1$$
as desired.

Let $x_r = x_{m_{i_1\cdots i_k}}$ with $1 \leq i_1< \cdots < i_k \leq s$
and $1 \leq k \leq s$. It is clear that
$$G_{S_{m_{i_1\cdots i_k},m}}\Big(\frac{x_m}{x_{m_{i_1\cdots i_k}}}\Big)
=\Big\{\frac{x_{m_{i_1}}}{x_{m_{i_1\cdots i_k}}},
\cdots, \frac{x_{m_{i_k}}}{x_{m_{i_1\cdots i_k}}}\Big\}$$
and
\begin{align*}
\delta\Big(\gcd\Big(J\cup\Big\{\frac{x_m}{x_{m_{i_1\cdots i_k}}}\Big\}\Big)\Big)
=\left\{\begin{aligned}
1\ \hbox{if}\ J=G_{S_{m_{i_1\cdots i_k},m}}\Big(\frac{x_m}
{x_{m_{i_1\cdots i_k}}}\Big),\\
0\ \hbox{if}\ J\subsetneq G_{S_{m_{i_1\cdots i_k},m}}
\Big(\frac{x_m}{x_{m_{i_1\cdots i_k}}}\Big).
\end{aligned}
\right.
\end{align*}
Hence by Lemma \ref{lemma 2.9} (iii), we have $c_{m_{i_1\cdots i_k}m}=(-1)^k$.
			
Now we treat with the case that $x_r \mid x_m$ with $r \neq m, m_{i_1\cdots i_k}$,
where  $1 \leq i_1< \cdots < i_k \leq s$ and $1 \leq k \leq s$.
For any subset $\bar{J}$ of $G_{S}^{(r)}(x_m)$,
we have $x_{m_{1\cdots s}}\mid \gcd(\bar{J}\cup \{x_m\})$ and
$x_r\mid \gcd(\bar{J}\cup \{x_m\})$. But
$$x_r\notin\{x_m\}\cup \{x_{m_{i_1\cdots i_k}}: 1\le k\le s, 1\le i_1<\cdots <i_k\le s\}.$$
Hence
$$x_r<x_{m_{1\cdots s}}\le \gcd(\bar{J}\cup \{x_m\})$$
and so $\frac{\gcd(\bar{J}\cup \{x_m\})}{x_r}>1$. This implies that
for any subset $\bar{J}$ of $G_{S}^{(r)}(x_m)$, we have
$\delta\big(\frac{\gcd(\bar{J}\cup \{x_m\})}{x_r}\big)=0$.
So by Lemma \ref{lemma 2.9} (iii), we have
\begin{align*}
c_{rm}=\sum\limits_{ \bar{J} \subseteq G_{S}^{(r)}(x_m)}(-1)^{|\bar{J}|}
\delta\Big(\frac{\gcd(\bar{J}\cup \{x_m\})}{x_{r}}\Big)
=\sum\limits_{ \bar{J} \subseteq G_{S}^{(r)}(x_m)}(-1)^{|\bar{J}|}\times0
=0
\end{align*}
as desired.

This finished the proof of Lemma \ref{lemma 2.10}.
\end{proof}
		
\begin{lemma}{\rm\cite{[ZY-2025]}}\label{lemma 3.6}
Let $S$ be a gcd-closed set satisfying the condition
$\mathcal G$ and let $x_m\in S$ be such that
$\{x_{m_1},\cdots,x_{m_s}\}\subseteq G_S(x_m)$.
For any real number $e$, let
$$
\beta_{\xi_e}(x_m):=x_m^e+\sum\limits_{t=1}^s(-1)^t\sum
\limits_{1\le i_1<\cdots<i_t\le s}x^e_{m_{i_1\cdots i_t}}.
$$
Then
\begin{align*}
\beta_{\xi_e}(x_m)=x^e_{m_{1\cdots s}}
\prod_{j=1}^s\Big(\Big(\frac{x_m}{x_{m_j}}\Big)^e-1\Big).
\end{align*}
\end{lemma}

\begin{lemma}\label{lemma 3.7}
Let $S$ be a gcd-closed set satisfies the condition
$\mathcal G$ with $x_m\in S$ and
$G_S(x_m)=\{x_{m_1},\cdots,x_{m_s}\}$.
Let $\alpha_{\xi_a}(x_m)$ be defined as in \eqref{eq 1.1}.
Then
$$\alpha_{\xi_a}(x_m)
=x_{m_{1\cdots s}}^a \prod_{j=1}^s\Big(\Big(\frac{x_m}
{x_{m_j}}\Big)^a-1\Big).$$
\end{lemma}
\begin{proof}
With Lemma \ref{lemma 2.8} applied to $f=\xi_a$,
by Lemma \ref{lemma 3.6}, we obtain that
\begin{align*}
\alpha_{\xi_a}(x_m)=\beta_{\xi_a}(x_m)
=&x^a_{m_{1\cdots s}}\prod_{j=1}^s
\Big(\Big(\frac{x_m}{x_{m_j}}\Big)^a-1\Big)
\end{align*}
as required. So Lemma \ref{lemma 3.7} is proved.
\end{proof}
		
\section{ Proof of Theorem \ref{theorem 1.3}}

For any integers $l$ and $m$ with $1\le l,m\le |S|$,
we define the two-variable functions $f$ and $g$ as follows:
\begin{align}\label{eq 4.1}
f(l,m):=\frac{1}{\alpha_{\xi_a}(x_m)}\sum_{x_r\mid x_m}c_{rm}(x_l,x_r)^b
\end{align}
and
\begin{align}\label{eq 4.2}
g(l,m):=\frac{1}{\alpha_{\xi_a}(x_m)}\sum_{x_r\mid x_m}c_{rm}[x_l,x_r]^b,
\end{align}
where $ \alpha_{\xi_a}(x_m)$ is defined as in \eqref{eq 1.1}.
Before giving the proof of Theorem \ref{theorem 1.3},
we show the following result.
		
\begin{lemma}\label{lemma 4.2}
Let $S$ be a gcd-closed set satisfying the condition $\mathcal{G}$ with $x_l,x_m\in S$,
and let $G_S(x_m)=\{x_{m_1},\cdots, x_{m_s}\}$ with $s\ge 1$. Let $a$ and $b$ be
positive integers. Then
$$f(l,m)=\left\{\begin{aligned}
&x_{m_{1\cdots s}}^{b-a}\prod_{j=1}^s
\frac{\big(\frac{x_m}{x_{m_j}}\big)^b-1}
{\big(\frac{x_m}{x_{m_j}}\big)^a-1}
&\ \hbox{if}\ x_m\mid x_l,\\
&0&\ \hbox{otherwise}
\end{aligned}\right.$$
and
$$g(l,m)=\left\{\begin{aligned}
&x_{m_{1\cdots s}}^{b-a}\Big(\frac{x_l}
{(x_l,x_m)}\Big)^b\prod_{j=1}^s
\frac{\big(\frac{x_m}{x_{m_j}}\big)^b-1}
{\big(\frac{x_m}{x_{m_j}}\big)^a-1}
&\ \hbox{if}\ (x_l,x_m)|x_{m_{1\cdots s}},\\
&0&\ \hbox{otherwise}.
\end{aligned}\right.$$
\end{lemma}
		
\begin{proof}
Since $G_S(x_m) = \{x_{m_{1}}, \cdots,x_{m_{s}}\}$,
by Lemma \ref{lemma 2.10}, \eqref{eq 4.1} and
\eqref{eq 4.2}, we have
\begin{align}\label{eq 4.3}
\alpha_{\xi_a}(x_m)f(l,m)=&\sum_{x_r\mid x_m}c_{rm}
(x_l,x_r)^b\notag\\
=&(x_l,x_m)^b+\sum\limits_{t=1}^s(-1)^t
\sum\limits_{1\le j_{1}<\cdots<j_{t}\le s}
(x_l,x_{m_{j_{1}\cdots j_{t}}})^b
\end{align}
and
\begin{align}\label{eq 4.4}
\alpha_{\xi_a}(x_m)g(l,m)=&\sum_{x_r\mid x_m}
c_{rm}[x_l,x_r]^b\notag\\
=&[x_l,x_m]^b+\sum\limits_{t=1}^s(-1)^t
\sum\limits_{1\le j_{1}<\cdots<j_{t}\le s}
[x_l,x_{m_{j_{1}\cdots j_{t}}}]^b.
\end{align}
Since $(x_l,x_m)|x_m$, it follows that either
$(x_l,x_m)|x_{m_{1\cdots s}}$, or $(x_l,x_m)=x_m $,
or there exist an integer $k$ with $1\le k\le s-1$
and $k$ integers $i_1,\cdots,i_k$ with
$1\le i_{1}<\cdots<i_{k}\le s$
such that $(x_l,x_m)|x_{m_{i_1\cdots i_k}}$
and $(x_l,x_m) \nmid x_{m_j} $ for all
$j \in\langle s\rangle \setminus\{i_{1},\cdots,i_{k}\}$.
So we only need to consider the following three cases.
			
{\sc Case 1.} $(x_l,x_m)|x_{m_{1\cdots s}}$. For
$1\le j_1<\cdots <j_t \le s$ with $1\le t\le s$,
since $x_{m_{1\cdots s}}| x_{m_{j_1\cdots j_t}}|x_m$,
we have $(x_l,x_m)|x_{m_{j_1\cdots j_t}}|x_m$ and
\begin{align}\label{eq 4.5}
(x_l,x_{m_{j_1\cdots j_t}})
=(x_l,(x_{m_{j_1\cdots j_t}}, x_m))
=((x_l,x_m),x_{m_{j_1\cdots j_t}})
=(x_l,x_m).
\end{align}
Then by \eqref{eq 4.3}, we have
\begin{align*}
\alpha_{\xi_a}(x_m)f(l,m)
=&(x_l,x_m)^b+\sum\limits_{t=1}^s(-1)^t\sum\limits_{1\le j_{1}<\cdots<j_{t}\le s}
(x_l,x_{m_{j_{1}\cdots j_{t}}})^b\\
=&(x_l,x_m)^b+\sum\limits_{t=1}^s(-1)^t\sum\limits_{1\le j_{1}<\cdots<j_{t}\le s}(x_l,x_m)^b\\
=&(x_l,x_m)^b\Big(1+\sum\limits_{t=1}^s(-1)^t\sum\limits_{1\le j_{1}<\cdots<j_{t}\le s}1\Big)\\
=&(x_l,x_m)^b\Big(1+\sum_{t=1}^s(-1)^t\binom{s}{t}\Big)\\
=&(x_l,x_m)^b(1-1)^s
=0.
\end{align*}
But $\alpha_{\xi_a}(x_m)\ne 0$. So $f(l,m)=0$.

On the other hand, since $S$ satisfies the condition $\mathcal G$, by Lemmas \ref{lemma 2.10} and
\ref{lemma 3.6}, identities \eqref{eq 4.4} and \eqref{eq 4.5}, we have
\begin{align*}
\alpha_{\xi_a}(x_m)g(l,m)
=&[x_l,x_m]^b+\sum\limits_{t=1}^s(-1)^t
\sum\limits_{1\le j_{1}<\cdots<j_{t}\le s}
[x_l,x_{m_{j_{1}\cdots j_{t}}}]^b\\
=&\dfrac{x_l^bx_m^b}{(x_l,x_m)^b}+\sum\limits_{t=1}^s(-1)^t
\sum\limits_{1\le j_{1}<\cdots<j_{t}\le s}
\dfrac{x_l^bx^b_{m_{j_1\cdots j_t}}}{(x_l,x_{m_{j_1\cdots j_t}})^b}\\
=&\dfrac{x_l^bx_m^b}{(x_l,x_m)^b} +\sum\limits_{t=1}^s(-1)^t
\sum\limits_{1\le j_{1}<\cdots<j_{t}\le s}
\dfrac{x_l^bx^b_{m_{j_1\cdots j_t}}}{(x_l,x_m)^b}\\
=&\dfrac{x_l^bx_m^b}{(x_l,x_m)^b}+\dfrac{x_l^b}{(x_l,x_m)^b}
\sum\limits_{t=1}^s(-1)^t\sum\limits_{1\le j_{1}<\cdots<j_{t}\le s}x^b_{m_{j_1\cdots j_t}}\\
=&\dfrac{x_l^b}{(x_l,x_m)^b}\Big(x_m^b+\sum\limits_{t=1}^s(-1)^t
\sum\limits_{1\le j_{1}<\cdots<j_{t}\le s}
x_{m_{j_1\cdots j_t}}^b\Big)\\				
=&\dfrac{x_l^b}{(x_l,x_m)^b}\beta_{\xi_b}(x_m)\\
=&\dfrac{x_l^bx_{m_{1\cdots s}}^b}{(x_l,x_m)^b}\prod_{j=1}^s
\Big(\Big(\frac{x_m}{x_{m_j}}\Big)^b-1\Big).
\end{align*}
			
Then applying Lemma \ref{lemma 3.7}, we obtain that
\begin{align*}
g(l,m)=& \frac{1}{\alpha_{\xi_a}(x_m)}\dfrac{x_l^bx_{m_{1\cdots s}}^b}{(x_l,x_m)^b}\prod_{j=1}^s
\Big(\Big(\frac{x_m}{x_{m_j}}\Big)^b-1\Big)\\
=& x_{m_{1\cdots s}}^{b-a}\Big(\frac{x_l}{(x_l,x_m)}\Big)^b\prod_{j=1}^s
\frac{\big(\frac{x_m}{x_{m_j}}\big)^b-1}{\big(\frac{x_m}{x_{m_j}}\big)^a-1}
\end{align*}
as expected.
		
{\sc Case 2.} $(x_l,x_m)=x_m$. Then $x_m|x_l$ and
so $[x_l,x_m]=x_l$. Clearly, for $1\le j_1<\cdots<j_t\le s$,
we have $x_{m_{j_1\cdots j_t}}|x_m|x_l$.
This infers that $(x_l, x_{m_{j_1\cdots j_t}})
=x_{m_{j_1\cdots j_t}}$ and $[x_l, x_{m_{j_1\cdots j_t}}]=x_l$.
Since $S$ satisfies the condition $\mathcal G$,
by \eqref{eq 4.3}, \eqref{eq 4.4}
and Lemma \ref{lemma 3.6}, we have
\begin{align*}
\alpha_{\xi_a}(x_m)f(l,m)=&\sum_{x_r\mid x_m}c_{rm}(x_l,x_r)^b\\
=&x_m^b +\sum\limits_{t=1}^s(-1)^t
\sum\limits_{1\le j_{1}<\cdots<j_{t}\le s}
(x_l, x_{m_{j_1\cdots j_t}})^b\\
=&x_m^b +\sum\limits_{t=1}^s(-1)^t
\sum\limits_{1\le j_{1}<\cdots<j_{t}\le s}
x_{m_{j_1\cdots j_t}}^b\nonumber\\
=&\beta_{\xi_b}(x_m)\\
=&x_{m_{1\cdots s}}^b\prod_{j=1}^s\Big(\Big(\frac{x_m}{x_{m_j}}\Big)^b-1\Big)
\end{align*}
and
\begin{align*}
\alpha_{\xi_a}(x_m)g(l,m)=&\sum_{x_r\mid x_m}c_{rm}[x_l,x_r]^b\\
=&x_l^b +\sum\limits_{t=1}^s(-1)^t
\sum\limits_{1\le j_{1}<\cdots<j_{t}\le s}
[x_l, x_{m_{j_1\cdots j_t}}]^b\nonumber\\
=&x_l^b +\sum\limits_{t=1}^s(-1)^t
\sum\limits_{1\le j_{1}<\cdots<j_{t}\le s}x_l^b\\
=&x_l^b\Big(1+\sum\limits_{t=1}^s(-1)^t
\sum\limits_{1\le j_{1}<\cdots<j_{t}\le s}1\Big)\nonumber\\
=&x_l^b(1-1)^s
=0.
\end{align*}
Hence by Lemma \ref{lemma 3.7}, one has
$$f(l,m)=\frac{1}
{\alpha_{\xi_a}(x_m)}x_{m_{1\cdots s}}^b\prod_{j=1}^s
\Big(\Big(\frac{x_m}{x_{m_j}}\Big)^b-1\Big)
=x_{m_{1\cdots s}}^{b-a}\prod_{j=1}^s
\frac{\big(\frac{x_m}{x_{m_j}}\big)^b-1}
{\big(\frac{x_m}{x_{m_j}}\big)^a-1}$$
and $$g(l,m)=\frac{1}{\alpha_{\xi_a}(x_m)}
\sum\limits_{x_r \mid x_m}c_{rm}[x_l,x_r]^b=0$$
as expected.
			
{\sc Case 3.} There exists an integer $k$ with $1\le k\le s-1$
and $k$ integers $i_1,\cdots, i_k$ with $1\le i_{1 }<\cdots<i_{k}\le s$
such that $(x_l,x_m)|x_{m_{i_1\cdots i_k}}$
and $(x_l,x_m) \nmid x_{m_j} $ for all
$j \in\langle s\rangle\setminus\{i_{1},\cdots,i_{k}\}$.
WLOG, one may let $(x_l,x_m)  \mid x_{m_{1\cdots k}}$
and $(x_l,x_m) \nmid x_{m_j} $ for all $k+1\le j\le s$.
Applying Lemma \ref{lemma 2.10} gives us that
\begin{small}
\begin{align}\label{eq 4.6}
&\sum_{x_r\mid x_m}c_{rm}(x_l,x_r)^b \nonumber\\
=&(x_l,x_m)^b+\sum_{h=1}^{k}(-1)^{h}\sum
\limits_{1\le i_1<\cdots<i_h\le k}(x_l,x_{m_{i_1\cdots i_h}})^b+\sum_{t=1}^{s-k}(-1)^{t}\sum
\limits_{k+1 \le j_1<\cdots <j_t \le s}
(x_l,x_{m_{j_1\cdots j_t}})^b\nonumber\\
&+\sum_{h=1}^{k}\sum_{t=1}^{s-k}(-1)^{h+t}
\sum\limits_{1 \le i_1<\cdots<i_h\le k
\atop k+1 \le j_1<\cdots <j_t \le s}
(x_l,x_{m_{i_1\cdots i_hj_1\cdots j_t}})^b\nonumber\\
=&C_1+C_2
\end{align}
\end{small}
and
\begin{small}
\begin{align}\label{eq 4.7}
&\sum_{x_r\mid x_m}c_{rm}[x_l,x_r]^b \nonumber\\
=&[x_l,x_m]^b+\sum_{t=1}^{s-k}(-1)^{t}\sum\limits_{k+1 \le j_1<\cdots <j_t \le s}[x_l,x_{m_{j_1\cdots j_t}}]^b+ \sum_{h=1}^{k}(-1)^{b}\sum\limits_{1 \le i_1<\cdots<i_h\le k}[x_l,x_{m_{i_1\cdots i_h}}]^b \nonumber\\
&+\sum_{h=1}^{k}\sum_{t=1}^{s-k}(-1)^{h+t}\sum\limits_{1 \le i_1<\cdots<i_h\le k \atop k+1\le j_1<\cdots <j_t \le s} [x_l,x_{m_{i_1\cdots i_hj_1\cdots j_t}}]^b\nonumber\\
=&D_1+D_2,
\end{align}
\end{small}
where
\begin{equation*}
C_1:=(x_l,x_m)^b+\sum_{h=1}^{k}(-1)^{h}\sum\limits_{1 \le i_1<\cdots<i_h\le k}(x_l,x_{m_{i_1\cdots i_h}})^b,
\end{equation*}
\begin{align*}
C_2:=&\sum_{t=1}^{s-k}(-1)^{t}\sum\limits_{k+1 \le j_1<\cdots <j_t \le s}(x_l,x_{m_{j_1\cdots j_t}})^b\notag\\
&+\sum_{h=1}^{k}\sum_{t=1}^{s-k}(-1)^{h+t}\sum\limits_{1 \le i_1<\cdots<i_h\le k \atop k+1 \le j_1<\cdots <j_t \le s} (x_l,x_{m_{i_1\cdots i_hj_1\cdots j_t}})^b,
\end{align*}
\begin{equation*}
D_1:=[x_l,x_m]^b+\sum_{t=1}^{s-k}(-1)^{t}\sum\limits_{k+1 \le j_1<\cdots <j_t \le s}[x_l,x_{m_{j_1\cdots j_t}}]^b
\end{equation*}
and
\begin{align*}
D_2:=&\sum_{h=1}^{k}(-1)^{h}\sum\limits_{1 \le i_1<\cdots<i_h\le k}[x_l,x_{m_{i_1\cdots i_h}}]^b\notag\\
&+\sum_{h=1}^{k}\sum_{t=1}^{s-k}(-1)^{h+t}\sum\limits_{1 \le i_1<\cdots<i_h\le k\atop k+1\le j_1 <\cdots <j_t \le s} [x_l,x_{m_{i_1\cdots i_hj_1\cdots j_t}}]^b.
\end{align*}
			
In the following we show that $C_1=C_2=D_1=D_2=0$.
First of all, we calculate $C_1$. By Lemma \ref{lemma 2.6} (i), we have
\begin{align}\label{eq 4.13}
C_1=&(x_l,x_m)^b+\sum_{h=1}^{k}(-1)^{h}\sum\limits_{1\le i_1<\cdots<i_h\le k}
(x_l,x_{m_{i_1\cdots i_h}})^b\notag\\
=&(x_l,x_m)^b+\sum_{h=1}^{k}(-1)^{h}\sum\limits_{1\le i_1<\cdots<i_h\le k}
(x_l,x_m)^b\notag\\
=&(x_l,x_m)^b\Big(1+\sum_{h=1}^{k}(-1)^{h}
\sum\limits_{1 \le i_1<\cdots<i_h\le k}1\Big)\notag\\
=&(x_l,x_m)^b\Big(1+\sum_{h=1}^{k}(-1)^{h}
\binom{k}{h}\Big)\notag\\
=&(x_l,x_m)^b(1-1)^k
=0
\end{align}
as desired.

Subsequently, we compute $C_2$. It follows from Lemma
\ref{lemma 2.6} (iv) that
\begin{align}\label{eq 4.19}
C_2=&\sum_{t=1}^{s-k}(-1)^{t}\sum\limits_{k+1 \le j_1<\cdots <j_t \le s}(x_l,x_{m_{j_1\cdots j_t}})^b\notag\\
&+\sum_{h=1}^{k}\sum_{t=1}^{s-k}(-1)^{h+t}
\sum\limits_{1\le i_1<\cdots<i_h\le k \atop k+1 \le j_1<\cdots <j_t\le s} (x_l,x_{m_{i_1\cdots i_hj_1\cdots j_t}})^b\notag\\
=&\sum_{t=1}^{s-k}(-1)^{t}\sum\limits_{k+1\le j_1<\cdots <j_t \le s}
(x_l,x_{m_{j_1\cdots j_t}})^b\notag\\
&+\sum_{h=1}^{k}\sum_{t=1}^{s-k}(-1)^{h+t}
\sum\limits_{1 \le i_1<\cdots<i_h\le k \atop k+1 \le j_1<\cdots <j_t \le s} (x_l,x_{m_{j_1\cdots j_t}})^b\notag\\
=&\sum_{t=1}^{s-k}(-1)^{t}\sum\limits_{k+1\le j_1<\cdots <j_t \le s}
(x_l,x_{m_{j_1\cdots j_t}})^b\Big(1+\sum_{h=1}^k(-1)^h
\sum\limits_{1\le i_1<\cdots<i_h\le k} 1\Big)\notag\\
=&\sum_{t=1}^{s-k}(-1)^{t}\sum\limits_{k+1\le j_1<\cdots <j_t \le s}
(x_l,x_{m_{j_1\cdots j_t}})^b\Big(1+\sum_{h=1}^k(-1)^h\binom{k}{h}\Big)\notag\\
=&\sum_{t=1}^{s-k}(-1)^{t}\sum\limits_{k+1\le j_1<\cdots <j_t \le s}
(x_l,x_{m_{j_1\cdots j_t}})^b(1-1)^k
=0
\end{align}
as required.

Therefore, by \eqref{eq 4.1}, \eqref{eq 4.6}, \eqref{eq 4.13}
and \eqref{eq 4.19}, we obtain that
$$f(l,m)=\dfrac{1}{\alpha_{\xi_a}(x_m)}
\sum\limits_{x_r \mid x_m}c_{rm}(x_l,x_r)^b=0.$$

Now we compute $D_1$. By Lemma \ref{lemma 2.6} (ii),
we deduce that
\begin{align}\label{eq 4.15}
D_1&=[x_l,x_m]^b+\sum_{t=1}^{s-k}(-1)^{t}
\sum\limits_{k+1\le j_1<\cdots <j_t \le s}
[x_l,x_{m_{j_1\cdots j_t}}]^b\notag\\
&=[x_l,x_m]^b+\sum_{t=1}^{s-k}(-1)^{t}
\sum\limits_{k+1\le j_1<\cdots<j_t\le s}[x_l,x_{m}]^b\notag\\
&= [x_l,x_m]^b\Big(1+\sum_{t=1}^{s-k}(-1)^{t}
\sum\limits_{k+1\le j_1<\cdots<j_t\le s} 1\Big)\notag\\
&= [x_l,x_m]^b\Big(1+\sum_{t=1}^{s-k}(-1)^{t}
\binom{s-k}{t}\Big)\notag\\
&=[x_l,x_m]^b\sum_{t=0}^{s-k}(-1)^{t}\binom{s-k}{t}\notag\\
&=[x_l,x_m]^b(1-1)^{s-k}
=0.
\end{align}

Consequently, we calculate $D_2$. By Lemma
\ref{lemma 2.6} (iii), we have
\begin{align}\label{eq 4.20}
D_2=&\sum_{h=1}^{k}(-1)^{h}\sum\limits_{1 \le i_1<\cdots<i_h\le k}
[x_l,x_{m_{i_1\cdots i_h}}]^b\notag\\
&+\sum_{h=1}^{k}\sum_{t=1}^{s-k}(-1)^{h+t}\sum\limits_{1\le i_1<\cdots<i_h\le k \atop k+1\le j_1<\cdots <j_t\le s}[x_l,x_{m_{i_1\cdots i_hj_1\cdots j_t}}]^b
\nonumber\\
=&\sum_{h=1}^{k}(-1)^{h}\sum\limits_{1 \le i_1<\cdots<i_h\le k}[x_l,x_{m_{i_1\cdots i_h}}]^b\notag\\
&+\sum_{h=1}^{k}\sum_{t=1}^{s-k}(-1)^{h+t}
\sum\limits_{1 \le i_1<\cdots<i_h\le k
\atop k+1\le j_1<\cdots <j_t \le s}
[x_l,x_{m_{i_1\cdots i_h}}]^b\nonumber\\
=&\sum_{h=1}^{k}(-1)^{h}\sum\limits_{1\le i_1<\cdots<i_h\le k}[x_l,x_{m_{i_1\cdots i_h}}]^b
\Big(1+\sum_{t=1}^{s-k}(-1)^{t}\sum\limits_{k+1\le j_1<\cdots <j_t \le s} 1 \Big)\nonumber\\
=&\sum_{h=1}^{k}(-1)^{h}\sum\limits_{1 \le i_1<\cdots<i_h\le k}[x_l,x_{m_{i_1\cdots i_h}}]^b
\Big(1+\sum_{t=1}^{s-k}(-1)^{t}\binom{s-k}{t}\Big)\nonumber\\
=&\sum_{h=1}^{k}(-1)^{h}\sum\limits_{1 \le i_1<\cdots<i_h\le k}[x_l,x_{m_{i_1\cdots i_h}}]^b(1-1)^{s-k}
=0.
\end{align}
It then follows from \eqref{eq 4.2}, \eqref{eq 4.7},
\eqref{eq 4.15} and \eqref{eq 4.20} that
$$g(l,m) =\dfrac{1}{\alpha_{\xi_a}(x_m)}
\sum\limits_{x_r \mid x_m}c_{rm}[x_l,x_r]^b=0.$$
So Lemma \ref{lemma 4.2} is proved.
\end{proof}
		
Finally, we present the proof of Theorem \ref{theorem 1.3}
as the conclusion of this paper.

\begin{proof}[Proof of Theorem \ref{theorem 1.3}]
Let $S$ be a gcd-closed set which satisfies the condition
$\mathcal{G}$. For any integers $i$ and $j$ with
$1\le i,j\le |S|$, by Lemma \ref{lemma 2.7} we have
\begin{align*}
{\left( (S^b)(S^a)^{-1}\right)}_{ij}
=&\sum_{r=1}^{|S|}(x_i,x_r)^b\sum\limits_{x_r\mid x_t\atop x_j\mid x_t} \dfrac{c_{rt}c_{jt}}{\alpha_{\xi_a}(x_t)}\\
=&\sum_{x_j\mid x_t}\dfrac{c_{jt}}{\alpha_{\xi_a}(x_t)}
\sum_{x_r\mid x_t} c_{rt}(x_i,x_r)^b\\
=&\sum_{x_j \mid x_t}c_{jt}f(i,t)
\end{align*}
and
\begin{align*}
{\left( [S^b](S^a)^{-1}\right)}_{ij}
=&\sum_{r=1}^{|S|}[x_i,x_r]^b\sum\limits_{x_r\mid x_t\atop x_j\mid x_t}
\dfrac{c_{rt}c_{jt}}{\alpha_{\xi_a}(x_t)}\\
=&\sum_{x_j\mid x_t}\dfrac{c_{jt}}{\alpha_{\xi_a}(x_t)}
\sum_{x_r\mid x_t} c_{rt}[x_i,x_r]^b\\
=&\sum_{x_j \mid x_t}c_{jt}g(i,t),
\end{align*}
where $f(i,t)$ and $g(i,t)$ are defined as in \eqref{eq 4.1}
and \eqref{eq 4.2}, respectively.
		
In what follows, we prove that
$${\left((S^b)(S^a)^{-1}\right)}_{ij}\in\mathbb{Z}$$
and
$${\left([S^b](S^a)^{-1}\right)}_{ij}\in \mathbb{Z}$$
for all integers $i$ and $j$ with $ 1 \le i,j \le |S|$. By the definition of
$c_{jt}$, we know that each $c_{jt}\in \mathbb{Z}$. So, in order to finish
the proof of Theorem \ref{theorem 1.3}, it is sufficient to show that
$f(i,t)\in \mathbb{Z}$ and $g(i,t)\in \mathbb{Z}$ for all $1\leq i,t\leq |S|$.
This will be done in what follows.
			
For any given integers $i$ and $t$ with $1\le i,t\le |S|$,
we consider the following two cases.
			
{\sc Case 1.} $t=1$. Then $x_t=x_1$. Since $S$ is
gcd closed, one has $x_1=\gcd(S)$. By \eqref{eq 1.1}, one has
$\alpha_{\xi_a}(x_1)=x_1^a$. Since $a|b$, from (\ref{eq 4.1}),
(\ref{eq 4.2}) and Lemma \ref{lemma 2.10}, we derive that
\begin{align*}
f(i,1)=\frac{1}{\alpha_{\xi_a}(x_1)}\sum_{x_r|x_1}c_{r1}(x_i,x_r)^b
=\frac{c_{11}(x_i,x_1)^b}{x_1^a}=\frac{x_1^b}{x_1^a}
=x_1^{b-a}\in \mathbb Z
\end{align*}
and
\begin{equation*}
g(i,1) = \dfrac{1}{\alpha_{\xi_a}(x_1)}\sum_{x_r\mid x_1} c_{r1}[x_i,x_r]^b
=\frac{c_{11}[x_i,x_1]^b}{x_1^a}=\frac{x_i^b}{x_1^a}=x_i^{b-a}\Big(\frac{x_i}{x_1}\Big)^a\in\mathbb Z		
\end{equation*}
as expected.

{\sc Case 2.} $t\ge 2$. Then $|G_S(x_t)|\ge 1$. Write
$G_S(x_t):=\{x_{t_1},\cdots, x_{t_s}\}$ with $s\ge 1$.
Since $a\mid b$ and $S$ satisfies the condition $\mathcal{G}$,
by Lemma \ref{lemma 4.2}, we have
$$f(i,t)=\left\{\begin{aligned}
&x_{t_{1\cdots s}}^{b-a}\prod_{j=1}^s
\frac{\big(\frac{x_t}{x_{t_j}}\big)^b-1}
{\big(\frac{x_t}{x_{t_j}}\big)^a-1}\in\mathbb Z
&\ \hbox{if}\ x_t\mid x_i,\\
&0&\ \hbox{otherwise}
\end{aligned}\right.$$
and
$$g(i,t)=\left\{\begin{aligned}
&x_{t_{1\cdots s}}^{b-a}\Big(\frac{x_i}
{(x_i,x_t)}\Big)^b\prod_{j=1}^s
\frac{\big(\frac{x_t}{x_{t_j}}\big)^b-1}
{\big(\frac{x_t}{x_{t_j}}\big)^a-1}\in\mathbb Z
&\ \hbox{if}\ (x_i,x_t)|x_{t_{1\cdots s}},\\
&0&\ \hbox{otherwise}.
\end{aligned}
\right.$$
Then the desired result that $f(i,t)\in\mathbb{Z}$
and $g(i,t)\in\mathbb{Z}$ follows immediately.
			
This concludes the proof of Theorem \ref{theorem 1.3}.
\end{proof}

\begin{center}
{\bf Acknowledgements}
\end{center}
The author would like to thank the editor and the anonymous referee
for their careful reading of the paper and very helpful suggestions,
comments and corrections that improve greatly the presentation of
this paper.

	\end{document}